\documentclass[11pt]{article}

\usepackage{amsmath}
\usepackage{amssymb}
\usepackage{amsthm}
\usepackage{graphicx}
\usepackage{amscd}
\usepackage{epic, eepic}
\usepackage{url}
\usepackage{color}
\usepackage{todonotes}
\usepackage{comment}
\usepackage{hyperref}
\usepackage[T1]{fontenc}
\usepackage[utf8]{inputenc}
\usepackage{mathtools}
\usepackage{paralist}

\newtheorem{theorem}{\bf Theorem}[section]
\newtheorem{lemma}[theorem]{\bf Lemma}

\newtheorem{obs}[theorem]{\bf Observation}
\newtheorem{problem}[theorem]{\bf Problem}
\newtheorem{prop}[theorem]{\bf Proposition}
\newtheorem{conj}[theorem]{\bf Conjecture}
\newtheorem{cor}[theorem]{\bf Corollary}
\newtheorem{const}[theorem]{\bf Construction}

\newtheorem{question}[theorem]{Question}

\theoremstyle{definition}

\newtheorem{remark}[theorem]{\bf Remark}
\newtheorem{defi}[theorem]{\bf Definition}

\newcommand{\cP}{\mathcal{P}}
\newcommand{\cA}{\mathcal{A}}

\newcommand{\cL}{\mathcal{L}}
\newcommand{\cB}{\mathcal{B}}

\newcommand{\cS}{\mathcal{S}}
\newcommand{\cK}{\mathcal{K}}

\newcommand{\cU}{\mathcal{U}}
\newcommand{\cH}{\mathcal{H}}

\newcommand{\PG}{\mathrm{PG}}

\newcommand{\GF}{\mathrm{GF}}

\newcommand{\pma}{\mathrm{pm}}
\newcommand{\Sp}{\mathrm{Sp}}

\newcommand{\F}{\mathbb{F}}

\usepackage{todonotes}

\title{Avoiding secants of given size in finite projective planes}

\author{Tamás Héger\thanks{Department of Computer Science and Information Theory, 
		Budapest University of Technology and Economics, M\H{u}egyetem rkp. 3., H-1111 Budapest, Hungary, and Department of Computer Science, ELTE E\"otv\"os Lor\'and University, H-1117 Budapest, P\'azm\'any P.\ stny.\ 1/C.
		E-mail: {\tt heger@cs.bme.hu, heger.tamas@ttk.elte.hu}}
	\and Zolt\'an L\'or\'ant Nagy\thanks{ELTE Linear Hypergraphs  Research Group,
		E\"otv\"os Lor\'and University, Budapest, Hungary. The author is supported by the Hungarian Research Grant (NKFIH) No. PD  134953. and K 124950.  	E-mail: {\tt nagyzoli@cs.elte.hu}}
}
\date{}

\begin{document}
	
	\maketitle
	
	\begin{abstract} Let $q$ be a prime power and $k$ be a natural number.
		What are the possible cardinalities of point sets  ${S}$
		in a  projective plane of order $q$,  which do not intersect any line at exactly $k$ points?
		%Given a positive integer $k$ and a point set $\mathcal{P}$ in a  projective plane of order $q$, for which no line intersect $\mathcal{P}$
		%in exactly $k$ points. How can we say about the size of $\mathcal{P}$? \\
		This problem and its variants have been investigated before, in relation with blocking sets, untouchable sets or sets of even type, among others. In this paper we show a series of results which point out the existence of all or almost all possible values $m\in [0, q^2+q+1]$ for $|S|=m$, provided that $k$ is not close to the extremal values $0$ or $q+1$.  Moreover, using polynomial techniques we show the existence of a point set $S$ with the following property: for every prescribed list of numbers $t_1, \ldots t_{q^2+q+1}$, $|S\cap \ell_i|\neq t_i$ holds for the $i$th line $\ell_i$,  $\forall i \in \{1, 2, \ldots, q^2+q+1\}$.  
	\end{abstract}
	
	\section{Introduction}

	Let $\Pi=\Pi_q$ denote a projective plane of order $q$ with point set $\cP$ and line set $\cL$. A  set $S$ of points in $\Pi$ is said to be of \textit{type }$(a_1,\ldots,a_r)$ if the set $\{a_i \colon  i=1\ldots r \}$ collects the line intersection numbers of the set $S$, i.e., $\{a_i \colon  i=1\ldots r \}=\{|S\cap \ell| \colon \ell\in \cL\}$.
	
	The most studied point sets $S$ are those with only two intersection numbers. In desarguesian planes $\PG(2,q)$ over the $q$-element field $\F_q$, these sets often have some algebraic structure, and their size $|S|$ is often well defined by the properties; we may mention $(k,n)$-arcs, Baer subplanes or unitals as some notable examples. 

	Sets with few intersection numbers in finite projective planes have been investigated by several authors, we refer to the works of Calderbank \cite{Calder}, Durante \cite{Durante}, Hirschfeld and Szőnyi \cite{HirSzonyi}, Brouwer, Coykendall and Dover     \cite{Brouwer, Dover}, Hamilton and Penttila \cite{Hami-Pent}. We also mention that the analogous problem has been also investigated in higher dimensional spaces as well, see e.g., \cite{Durante, Inna, Zanetti}.
	
	However, not much seems to be known in the general case of  type sets with cardinality  more than three; see Hirschfeld \cite{Hirsch} for a survey of known results.
	
	Our aim is to initiate the investigation of  sets $S$, for which some fixed value $k\in \{0,1,\ldots, q+1\}$ is \textit{missing from the type} of $S$, and discuss whether such sets exist with every possible  size $|S|=m\in  \{0,1,\ldots, q^2+q+1\}$.
	
	\begin{defi}[Secants, $k$-avoiding property]
		Let $S$ be a set of points in a projective plane $\Pi_q$ of order $q$. A line $\ell$ is called \emph{skew}, \emph{tangent}, or \emph{$k$-secant} to $S$, if $|S\cap \ell|$ equals $0$, $1$, or $k$, respectively. $S$ is \emph{$k$-avoiding}, if $S$ has no $k$-secants.
	\end{defi}
	
	For specific values of $k$, this problem has been studied before. For $k=0$, point sets without skew lines are called \emph{blocking sets}, and these have been subject to extensive study since the middle of the last century; for a survey on blocking sets, see \cite{BlokhuisSzonyi} or \cite{KSz}. In the present context, the problem is to determine the possible sizes of blocking sets, which is well understood and quite trivial anyways: a blocking set of size $m$ exists in $\Pi_q$ if and only if $m\geq q+1$. This follows from the fact that the smallest blocking sets are lines, and that any superset of a blocking set is also a blocking set.
	
	The case $k=1$ requires sets avoiding tangent lines, and this case is already far from being trivial. Such sets were called \emph{sets without tangents} or \emph{untouchable sets}, see the papers of
	Blokhuis, Seress, Wilbrink \cite{Blokhuis}, and Blokhuis, Szőnyi, Weiner \cite{BlokhuisSzonyiW}. It is clear that such a set, unless empty, must have at least $q+2$ points (consider a point of the set and the $q+1$ lines passing through it). When $q$ is even, in $\PG(2,q)$ there exist sets without tangents of size $q+2$, called \emph{hyperovals}. Yet, as the superset of a $1$-avoding set is not necessarily $1$-avoiding, this observation does not close the case. When $q$ is odd, Blokhuis, Seress, Wilbrink showed via polynomial techniques that $1$-secants always exist for sets $S$ if $0<|S|\leq q+0.25\sqrt{2q}+1$ holds. 
	
	Another motivation comes from the Euclidean geometry. The celebrated Sylvester-Gallai theorem states that a finite non-collinear point set $S$ on the Euclidean plane always determines a 2-secant. In other words, if we take all the lines of the geometry and consider their intersection with $S$, the intersection size 2 is unavoidable (unless $S$ is a subset of a line). Embedding $2$-secant free, so-called Sylvester-Gallai designs has been studied even in the '50s by Motzkin \cite{Motzkin}, and later by Kelly and Nwankpa \cite{sylgall}.  
	
	Our problem can be generalized to higher dimensions as well, where one seeks point sets of given size which does not have subspace intersections of prescribed cardinality. The famous cap set problem \cite{Ellen} can be rephrased also in such a way, here we refer to the recent paper of Kovács and Nagy \cite{KN1, KN2} on the higher dimensional results and the detailed exposition of related topics, such as extremal graph theoretic analogues \cite{EFRS}.

	We formalize the problem to be discussed in the present paper as follows. 
	
	%Let $G(n,e)$ denote a graph on $n$ vertices and $e$ edges.
	%Erdős, Füredi, Rothschild and T. Sós initiated the investigation of the following problem \cite{EFRS}. Fix a positive integer $m$ and a pair of integers $(n, e)$ such that $0\leq e\leq \binom{n}{2}$. For which $f$ does it hold that 
	%any $n$-vertex graph with $e$ edges contains an induced subgraph on $m$ vertices having $f$ edges? Equivalently, we are seeking pairs $(m,f)$ such that $m$-vertex subgraphs with $f$ edges are unavoidable in graphs of the form $G(n,e)$.
	
	% Erdős, Füredi, Rothschild and Sós introduced the notation  $(n,e)\rightarrow (m,f)$ for induced $(m,f)$-subgraphs being guaranteed in graphs $G(n,e)$.
	
	%The main result of their paper from 1999 showed that forced pairs $(m,f)$ are rare in the following sense. Consider the set $S(n,m,f)$ which contains all edge cardinalities $e$ s.t.  $(n,e)\rightarrow (m,f)$.
	%Its density is the ratio $\frac{|S(n,m,f)|}{\binom{n}{2}}$. It was proven that the limit superior of this density is bounded from above by $2/3$ apart from a handful cases, and is $0$ for the majority of the pairs $(m,f)$. Erdős, Füredi, Rothschild and T. Sós
	%conjectured that in fact it is bounded from above by $1/2$ apart from finitely many pairs $(m,f)$. This was confirmed recently by  He,  Ma and  Zhao \cite{He}. Several related problems have been studied in the last couple of years concerning graph and hypergraph theoretic settings \cite{Axen, Axen2, Zarb} and vector space setting \cite{KN1, KN2}.
	
	\begin{problem}\label{mainproblem}
		Let us fix a projective plane $\Pi_q$ of order $q$ and an integer $k\in[0,q+1]$.
		\begin{itemize}
			\item Given an integer $m\in [0,q^2+q+1]$, is it true that there exists a $k$-avoiding point set $S$ in $\Pi_q$ of size $m$?
			\item Determine the spectrum  $\Sp(k,\Pi_q)$, the  set of integers $m$ for which a $k$-avoiding point set of size $m$ exists in $\Pi_q$. If $\Pi_q$ is  desarguesian, we may use the notation $\Sp(k,q)$ instead of $\Sp(k,\PG(2,q))$. 
			%  \item Determine the limit superior of the density $\limsup_{q}|Sp(k,q)|/(q^2+q+1)$  for fixed values of $k$, or for $k$ which is some functions of $q$.
		\end{itemize}
	\end{problem}

	\begin{obs} 
		As the complement of a $k$-avoiding set $S$ is clearly $(q+1-k)$-avoiding, it follows that $\Sp(q+1-k,\Pi_q)= \{q^2+q+1-m\colon m\in \Sp(k,\Pi_q)\}$. 
	\end{obs}
	
	Consequently, we may suppose that $k\leq \frac{q+1}{2}$, and this will be assumed from now on.
	
	It appears that the bottleneck of the problem is to decide whether $m\in \Sp(k,\Pi_q)$ for those cardinalities $m$ where the average intersection size $\frac{q+1}{q^2+q+1}\cdot m$ is close to $k$. The result below, showing  $m\in \Sp(k,\Pi_q)$ for every $m$ for which $|\frac{q+1}{q^2+q+1}\cdot m - k|\geq 1$, follows from a simple construction.
	
	\begin{prop}[Critical window]\label{trivikonst}
		Let $\Pi_q$ be an arbitrary projective plane of order $q\geq 4$, and let $1\leq k\leq (q+1)/2$. Then we have the following:
		\begin{itemize}
			\item If $k\geq 3$, then $[0,(k-1)q+1]\cup [(k+1)q,q^2+q+1]\subset\Sp(k,\Pi_q)$. 
			\item If $k=1$, then $(\{2q\}\cup [2q+2,q^2+q+1])\subset\Sp(1,\Pi_q)$, and $[0,q+1]\cap\Sp(1,\Pi_q)=\emptyset$.
			\item If $k=2$, then $([0,q+1]\setminus\{2\})\cup [3q,q^2+q+1]\subset\Sp(2,\Pi_q)$, and $2\notin \Sp(2,\Pi_q)$.
		\end{itemize}
	\end{prop}
	
	In light of the above proposition, we refer to the interval $[(k-1)q+2, (k+1)q-1]$ as the \textit{critical window}. Exhaustive search suggests that for small values of $q$, several elements of the critical window are missing from the spectrum, provided that $k$ is rather small; see Table \ref{tab:smallq} in Section \ref{sec3}  for details. The case $k=1$ is also discussed in detail in Section \ref{UTS}.

	Our main results concerning Problem \ref{mainproblem} are as follows.
	
	\begin{theorem}\label{maxiv} Let $q>2$ be even.  Then 
		$[kq+4, (k+1)q+1]\subset \Sp(k,q)$ hence $|Sp(k,q)|\geq q^2-3$,  provided that $3\leq k\leq q/4+3$ holds.
	\end{theorem}
	
	If $q$ is a large enough square number, there are $m$-sets where $k$-secants are avoidable  for every  $m$, provided that $k$ is not too small.
	
	\begin{theorem}\label{baeres}
		Let $q\geq 25$ be a square prime power, $2\sqrt{q}+2\leq k \leq q-2\sqrt{q}-1$. Then there exists an $m$-set of points avoiding $k$-secants for all $m$ with $0\leq m\leq q^2+q+1$.
	\end{theorem}
	
	A natural extension of Problem \ref{mainproblem} is to prescribe values for each line $\ell\in \cL(\Pi_q)$ which are forbidden secant sizes and ask for the existence of projective  point sets which satisfy each requirement. 
	Let us introduce the following notations.
	\begin{defi}
		Let $f: \mathcal{L}\to \mathbb{N}$ be a function which assigns a non-negative integer to every line $\ell\in \mathcal{L}$ of a projective plane $\Pi_q$ of order $q$. Here we do not require that the plane $\Pi_q$ is desarguesian.  An \textit{$f$-avoiding set $S$   in $\Pi_q$} is a point set  where for  each line $\ell\in \mathcal{L}$, $|S\cap \ell| \neq f(\ell)$ holds.    
	\end{defi}
	
	For \textit{$f$-avoiding set $S$   in $\Pi_q$}, we have the following theorem.
	\begin{theorem}\label{nullst}
		There exists an $f$-avoiding set in  every projective plane $\Pi_q$  for every  function $f$.
	\end{theorem}
	
	The paper is organised as follows. %\todo{TBF}
	In Section \ref{sec2} we prove Theorem \ref{nullst} using Alon's Combinatorial Nullstellensatz, moreover we show its modular extension and discuss the question of sharpness in terms of the number of excluded intersection numbers.
	Next we show that values outside the critical value can always be admitted by a suitable $k$-avoiding point set. Furthermore we study the case when $k$ is rather small, classify and extend previously examined construction schemes and present numerical results. These support the conjecture that while several values are missing from the support within the critical window, but only if $k$ is very small (provided that $k\leq \frac{q+1}{2}$). \\
	In Section \ref{sec:4}, we show a series of constructions to prove Theorem \ref{maxiv} and \ref{baeres} and some further results. These are usually obtained from group theory results and some set operations, using point sets such a type which does not contain line intersection numbers close to $k$. Prominent examples for such sets are lines, Baer subplanes, maximal arcs and unitals. Note however that one might assume further conditions (e.g., the order $q$ being a square) for the existence of such objects.
	Finally, we end our paper by pointing out some connection points in coding theory and formulate several conjectures.

	\section{Avoiding prescribed secant sizes for each line}\label{sec2}
	
	Recall that  $f: \mathcal{L}\to \mathbb{N}$ is a function which assigns a non-negative integer to every line $\ell\in \mathcal{L}$ of a projective plane $\Pi_q$ of order $q$.   An \textit{$f$-avoiding set $S$   in the projective plane} is such that for each line $\ell\in \mathcal{L}$, $|S\cap \ell| \neq f(\ell)$.
	
	Here we prove Theorem \ref{nullst} which claims the existence of an $f$-avoiding set in $\Pi_q$  for every  function $f$.
	For the proof of this result we apply the so-called Non-vanishing lemma, which is the consequence of Alon's Combinatorial Nullstellensatz.
	
	\begin{lemma}[Combinatorial Nullstellensatz, Non-vanishing lemma \cite{Alon}]\label{nonvanish}  Let $\mathbb{F}$ be an arbitrary field and let $P=P(x_1, \ldots, x_k)$ be a polynomial of $k$ variables over  $\mathbb{F}$. Suppose that there exists a monomial $\prod_{i=1}^k {x_i^{d_i}}$, such that the sum  $\sum_{i=1}^k {{d_i}}$ equals the total degree of $P$, and the coefficient of $\prod_{i=1}^k {x_i^{d_i}}$ in $P$ is nonzero. Then for any set of subsets $A_1, \ldots, A_k $ of $ \mathbb{F}$ such that $ |A_i| >d_i$, there exists a $k$-tuple $(s_1, s_2, \ldots, s_k)\in \bigtimes_{i=1}^k A_i$ for which $P(s_1, s_2, \ldots, s_k)\neq 0$.
	\end{lemma}

	\begin{proof}[Proof of Theorem \ref{nullst}] Let us assign a variable $x_i$ to each point $P_i$ of the plane $\Pi_q$. For each line $\ell$ we assign a linear term $(\sum_{P_i\in \ell} x_i-f(\ell))$.
		
		Consider the polynomial $$Q(x_1,\ldots, x_{q^2+q+1})=\prod_{\ell\in \mathcal{L}} (\sum_{P_i\in \ell} x_i-f(\ell))$$  over $\mathbb{R}$. We clearly have $\deg(Q)=q^2+q+1$. Moreover, $\prod_{i=1}^{q^2+q+1} x_i$ is a monomial of degree equal to $\deg(Q)$ with nonzero coefficient. Indeed, the coefficient of the 
		monomial in view equals the number of ways one can assign an incident point to every line such that the assigned points are pairwise distinct, i.e., the number of perfect matchings in the incidence graph of the projective plane. Since the incidence graph is regular, the number of perfect matchings is positive in view of Hall's Marriage Theorem. Now let us choose $A_i=\{0,1\}$ for all $i$. Then a non-vanishing element in $\bigtimes_{i=1}^{q^2+q+1} A_i$ is a characteristic vector of an $f$-avoiding set.
	\end{proof}

	We can also consider the modular version of Theorem \ref{nullst}. Suppose that a prime $p$ does not divide the permanent of the adjacency matrix of the Levi graph of $\Pi_q$, which  equals the number of perfect matchings $\pma(\Pi_q)$ in the incidence graph of $\Pi_q$. Then  we have the following.%\todo{$\pma$-ról mondani vmit}
	
	\begin{theorem}\label{nullst2}
		For every prime $p$ for which $p  \nmid \pma(\Pi_q),$ there exists  a set $S$  such that for each line $\ell$, $|S\cap \ell| \neq f(\ell) \pmod p$
	\end{theorem}
	%\todo{VargaLaci cikkét megnézni, hátha primhatvanyra is megy}
	\begin{proof}
		The proof follows the proof of Theorem \ref{nullst} after defining the same polynomial but over a finite field $\mathbb{F}_p$, and observing that the required nonzero coefficient  of $\prod_{i=1}^{q^2+q+1} x_i$ is provided by the property $p  \nmid\pma(\Pi_q)$.
	\end{proof}
	
	Unfortunately, neither any general formula, nor general divisibility properties are known for $\pma(\Pi_q)$. We refer to 
	\cite{perma1} concerning this problem. %\perma2
	
	It is natural to consider the following generalisation.
	
	\begin{problem}
		Suppose that on each line $\ell\in \mathcal{L}$, a set of forbidden values is given by a function $f^*:\mathcal{L}\to 2^{\mathbb{N}}$. What is the minimum of the sum of the cardinalities of the forbidden values  for the set of all lines $\sum_{\ell\in \mathcal{L}}|f^*(\ell)|$ for which there is no $f^*$-avoiding set in $\Pi_q$, provided that $|f^*(\ell)|\leq q$ for all $\ell$?
	\end{problem}
	The condition $|f^*(\ell)|\leq q$ simply ensures that at least two admissible values remain on each line. Having only one admissible value for a pair of lines would enable us to require that $|S\cap \ell|=0, |S\cap \ell'|=q+1$ for an admissible set $S$ and two distinct lines $\ell, \ell'$. But such a set $S$ does not exist since  $\ell\cap\ell'\not \in S$ and   $\ell\cap\ell' \in S$ would be required at the same time.
	
	Next we show that the above algebraic method provides a sharp result in the sense that even  in total $q^2+q+2$ forbidden values $ \sum_{\ell\in \mathcal{L}}|f^*(\ell)|$  cannot be met in general.
	
	\begin{const}\label{q2+q+2}
		Take a set $\mathcal{L}_P:=\{\ell: P\in \ell\}$ of lines through a given point $P$. Let $f^*(\ell):=\{1,2,\ldots, q\}$ for all $\ell \in \mathcal{L}_P$.
		Take two (not necessarily distinct) lines $\ell', \ell'' \in \mathcal{L}\setminus \mathcal{L}_P$ and set $f^*(\ell'):=\{0\}$ , $f^*(\ell''):=\{q+1\}$. 
	\end{const} 
	
	\begin{prop} There is no $f^*$-avoiding point set for  Construction \ref{q2+q+2}.
	\end{prop}
	
	\begin{proof} Suppose by contradiction the existence of such a set. The conditions on  $f^*(\ell): \ell \in \mathcal{L}_P$ ensures that 
		every  line $\ell \in \mathcal{L}_P$ is either contained or is disjoint from an   $f^*$-avoiding point set. Observe that if $P$ is in the $f^*$-avoiding point set, then then this implies that the set is the point set of $\Pi_q$ while if $P$ is not in the set then it should be the empty set. However, the conditions $f^*(\ell'):=\{0\}$, $f^*(\ell''):=\{q+1\}$ exclude both possibilities.
	\end{proof}
	
	Instead of a global bound on $\sum_{\ell\in \mathcal{L}}|f^*(\ell)|$, one may impose local (upper) bounds on  $|f^*(\ell)|$ and investigate whether these conditions can be met to find  a $f^*$-avoiding point set.
	
	\begin{problem}\label{fstaravoiding} Determine the maximum value $M$ for which the following holds. For every  function $f^*:\mathcal{L}\to 2^{\mathbb{N}}$ which satisfies $|f^*(\ell)|\leq M:=M(q)$ for all $\ell$, there exists an $f^*$-avoiding set in $\Pi_q$.
	\end{problem}
	
	Finally, let us mention that for any  $g:\mathcal{L}\to {\mathbb{N}}$, there exists at most one point set $S$ such that for each line $\ell\in \mathcal{L}$, $|S\cap \ell| = g(\ell)$, that is, a set is characterised via its intersection pattern. 
	Let us denote the incidence matrix of the projective plane by $M$, which is well know to be a non-singular matrix. The claim above follows from the fact that a characteristic vector $\textbf{v}$ of a set 
	will be a unique solution to the equality $M^T\textbf{v}=\textbf{g}$, where $\textbf{g}=(g(\ell_1), g(\ell_2), \ldots, g(\ell_{q^2+q+1}))$ is the vector obtained by the prescribed intersection numbers. 
	
	%   \todo[inline]{ OTLET: VEGYÜNK EGY $\mathcal{P}$ PONTHALMAZT BE, vagyis ezeken az $x_i=1$ rögfzítve van, de csináljuk ezt úgy, hogy sok sok egyenesen a kizárt méretet automatikusan kerüljük el. (pl. ovális: a kizárt 1-es méret a felére el lett kerülve. $k$ db Baer: a kizárt t méret el lett kerülve $kq$ egyenesen ha $t<\sqrt{q}+k$. Így bevehetünk még néhány $(\sum_i x_i-m_j)$ tényezőt)}

	\section{The critical window}\label{sec3}
	
	Recall that the critical window is $[(k-1)q+2, (k+1)q-1]$ with respect to $k$ for projective planes $\Pi_q$, which is claimed to consists of the crucial values in order to determine the spectrum  $\Sp(k,\Pi_q)$.
	
	\subsection{Out of the critical window}\label{label34}
	
	First we show that  given $k, q$, there exist suitable point sets in an arbitrary projective plane $\Pi_q$ of order $q$ of size $m$ which do not have $k$-secants for every $m$ outside the critical window $[q(k-1), q(k+1)+1]$. This relies on the perturbation of a pencil of $s\approx m/q$ lines. Typically, one may take a pencil of $\lfloor m/q\rfloor +1$ lines and remove a suitable amount of points from one of these lines to obtain an appropriate point set. However, as this would not work in some particular cases, we will use two partial lines instead of one.
	
	\begin{const}\label{triviconst}
		Let $k\geq 1$. Take a point $P$ in $\Pi_q$, let $\ell_1, \ldots \ell_{s}$ be $s$ lines through $P$ for some $2\leq s\leq q+1$, and for $i\in\{1,2\}$, let $X_i\subset \ell_{i}\setminus \{P\}$ be a set of size $|X_i|=t_i$, where $0\leq t_i\leq q$. Let 
		\[S = \cup_{i=3}^{s}(\ell_i\setminus\{P\}) \cup X_1\cup X_2.\]
	\end{const}
	
	\begin{prop}\label{triviconstprop}
		Consider the set $S$ in Construction \ref{triviconst}. Then $|S|=(s-2)q+t_1+t_2$, and for any line $\ell$ of $\Pi_q$, $|S\cap \ell|\in\{0,t_1,t_2,s-2,s-1,s,q\}$.
	\end{prop}
	\begin{proof}
		The size of $S$ is self evident. Let $\ell$ be any line of $\Pi_q$. If $P\in\ell$, then $|S\cap \ell|$ is $t_i$, $q$ or $0$ according to whether $\ell\in\{\ell_1,\ell_2\}$, $\ell=\ell_i$ for some $3\leq i\leq s$, or none of the two previous cases, respectively. If $P\notin\ell$, then $\ell$ meets each $\ell_i$ in a unique point ($1\leq i\leq s$), and at most two of these may be not in $S$, giving $s-2\leq |S\cap\ell|\leq s$.
	\end{proof}
	
	From this, we get the
	
	\begin{proof}[Proof of Proposition \ref{trivikonst}] 
		As observed in the Introduction, we may assume $1\leq k\leq (q+1)/2$.
		Consider the point set $S$ defined in Construction \ref{triviconst}. By Proposition \ref{triviconstprop}, $S$ does not have a $k$-secant if $s\leq k-1$ or $s\geq k+3$ and, furthermore, $k\notin\{t_1,t_2\}$. 
		
		Note that if $k\leq 2$, then $2\leq s\leq k-1$ is not possible. However, for $k=1$, we only need to prove the existence of $1$-avoiding sets of size $m\geq 2q$. If $k=2$, then $m\neq 2$ collinear points ($0\leq m\leq q+1$) provides a $2$-avoiding set of size $m$ (and, clearly, no $2$-avoiding set of size $2$ exists). If $k\geq 3$, Construction \ref{triviconst} can provide examples in the lower region.
		
		Suppose now we want to achieve the size $|S|=(s-2)q+x$ with $0\leq x\leq 2q$ for some $s\in\{2,\ldots,k-1,k+3,\ldots,q+1\}$. If $x=0$ or $x=2q$, $t_1=t_2=0$ or $t_1=t_2=q$ are suitable as $1\leq k\leq (q+1)/2$. If $x=1$, we can only choose $\{t_1,t_2\}=\{0,1\}$, which is not suitable if and only if $k=1$. If $x=2q-1$, then $\{t_1,t_2\}=\{q-1,q\}$. As $k\leq (q+1)/2$, this may cause a problem only if $k=q-1$ and $q-1\leq (q+1)/2$, that is, $q\leq 3$, which is ruled out by our restriction $q\geq 4$. If $2\leq x\leq 2q-2$, then we have at least two options for $\{t_1,t_2\}$, at least one of which will be suitable, namely $t_1=\lfloor x/2\rfloor$ and $t_2=\lfloor x/2\rfloor$, or $t_1=\lfloor x/2\rfloor-1$ and $t_2=\lfloor x/2\rfloor+1$. Hence the only possibly not achievable size in the interval $[(s-2)q,sq]$ remains $|S|=(s-2)q+1$ for $k=1$. However, to achieve this size, we may decrease $s$ by one and set $x=q+1$, unless $s=k+3=4$. This means that for $k=1$, the only size we can not achieve is $|S|=2q+1$. 
		
		Thus, if $q\geq 4$, taking all possible values of $s$ we get that $[0,(k-1)q]\cup[(k+1)q,(q+1)q] \subset \Sp(k,\Pi_q)$ (with the only exceptions for $k=1$ and $|S|=2q+1$, and $k=2$ and $|S|=2$). Clearly, $q^2+q+1\in\Sp(k,\Pi_q)$ holds as well. Moreover, if $k\geq 2$, we may take the union of the point sets of $k-1$ lines passing through a point $P$. This set is of size $(k-1)q+1$ and every line meets it in $1$, $k-1$ or $q+1$ points, hence it is $k$-avoiding. Summarizing these results, we get the desired statement.
		\iffalse
		
		First suppose that $k>1$.
		
		Take the point set of $k-1$ concurrent lines $\ell_1, \ldots \ell_{k-1}$, which has $q(k-1)+1$ points. Delete some points to have $m$ remaining points. Note that after the deletion, every line not contained in $\{\ell_1, \ldots \ell_{k-1}\}$ intersects the point set in less than $k$ points. Thus if one can perform the deletion so that each of the lines $\{\ell_1, \ldots \ell_{k-1}\}$ contain different number of points than $k$, we are done. 
		
		The case $m\ge q(k+1)$ is fairly similar.
		Take the point set  of $k+1$ concurrent lines $\ell_1, \ldots \ell_{k+1}$ and delete their common intersection point $P$. This point set  has $q(k+1)$ points.  Add some points to have $m$ points in total. Note that after the addition, every line in the set $\{\ell_1, \ldots \ell_{k-1}\}$  and every line $\ell$ for which $P \not \in \ell$  intersects the point set in more than $k$ points. Thus one only have to make sure that the remaining lines through $P$ contain different number of points than $k$. This can be easily achieved unless $k=1$, when the addition of a further point guarantees the existence of a tangent line.
		
		\fi
	\end{proof}
	
	\subsection{Missing values in the critical window when $k$ is small}
	
	Let us mention in advance that we expect several values $m$ in the critical window which are missing from the spectrum $\Sp(k,\Pi_q)$ if $k$ is relatively small.  We elaborate on previously known results and some new ones below.
	
	\subsubsection{$k=0$, blocking sets}
	
	Recall that for $k=0$, determining the spectrum is trivial: any set of at most $q$ points has at least one skew line, furthermore, any set containing a line is $0$-avoiding (or, as called commonly, a blocking set), hence $\Sp(0,\Pi_q)=[q+1,q^2+q+1]$. Thus, the values in the upper half of the critical window are indeed missing.
	
	\subsubsection{$k=1$, untouchable sets}\label{UTS}
	
	The study of $1$-avoiding (or untouchable) sets was initiated in \cite{Blokhuis}, and the focus was put on finding the smallest possible size of such a set (which is interesting only if $q$ is odd, hence untouchable sets are often studied under this assumption on $q$). Let us collect some results about these which are relevant to Problem \ref{mainproblem}. A \emph{set of even type} is a set of points which meets every line in an even number of points. Clearly, a set of even type is $1$-avoiding. It is easy to show that if there is a set of even type in $\Pi_q$, then $q$ must be even.
	
	\begin{theorem}[Lower bounds on the size of $1$-avoiding sets]
		Let $S$ be a $1$-avoiding set in the desarguesian projective plane $\PG(2,q)$. Then the following hold.
		\begin{enumerate}
			\item $|S|\geq q+2$, and an example of  size $q+2$ exists if and only if $q$ is even (folklore). These examples are \emph{hyperovals}, which are sets of type $(0,2)$.
			\item If $q$ is odd, then $|S|>q+\frac14\sqrt{2q}+2$ \cite{Blokhuis}. If $q$ is even and $S$ is not of even type, then $|S|\geq q+\sqrt{q/6}+1$ \cite{BlokhuisSzonyiW}.
			\item The size $u_q$ of the smallest $1$-avoiding set in $\PG(2,q)$, $q$ odd, is the following \cite{Blokhuis, Geertrui}:\\
			\begin{tabular}{c|c|c|c|c|c}
				$q$   &  3  &  5  &  7  &  9  &  11    \\ \hline
				$u_q$ &  6  &  10 &  12 & 15  &  18
			\end{tabular}
		\end{enumerate}   
	\end{theorem}
	
	Next we give several constructions found in the literature in a somewhat unified manner, which suggests the upcoming generalisation of these examples. We will need several notions regarding blocking sets which we briefly recall here; for  more information about these subjects, we refer to \cite{KSz} or \cite{Hirsch}.
	
	A \emph{blocking set} in a projective plane $\Pi_q$ is a set $\cB$  of points which intersects every line. A blocking set $\cB$ is of \emph{Rédei type} if $|\cB|=q+k$ ($k\leq q$) admitting a $k$-secant line, called a \emph{Rédei line of $\cB$}. Equivalently, Rédei type blocking sets in $\PG(2,q)$ can be described as the union of the graph of a function $f\colon\GF(q)\to\GF(q)$ and the set of directions determined by $f$. A \emph{Baer subplane} in $\Pi_q$ is a subplane of order $\sqrt{q}$. It is a blocking set of size $q+\sqrt{q}+1$. By Bruen's well-known result, we have that in any projective plane of order $q$, any blocking set of size $q+\sqrt{q}+1$ either contains a line or is a Baer subplane. As Baer subplanes admit $(\sqrt{q}+1)$-secants, they are also Rédei type blocking sets. Let us remark that lines are also Rédei type blocking sets. A blocking set is called \emph{small} if it has less than $3(q+1)/2$ points.
	
	\begin{theorem}[Constructions of $1$-avoiding sets]\label{UTSconstr}
		In $\PG(2,q)$, there is a $1$-avoiding set of the following size:
		\begin{enumerate}
			\item $2q$; this is the symmetric difference of two lines \cite{Blokhuis}, also covered by Construction \ref{triviconst}.
			\item $2q-2$ if $q\geq 7$; this is the symmetric difference of two carefully chosen ovals which have exactly two common points \cite{Blokhuis}.
			\item $2q-\sqrt{q}$, if $q$ is a square; this is the symmetric difference of a Baer subplane and a line \cite{Blokhuis}.
			\item $2q-\frac{q}{p}$ \cite{Blokhuis} and $2q-\frac{q-p}{p-1}$ \cite{Geertrui, LSVdV}\footnote{Let us remark that Theorem 6 of \cite{Geertrui} is mistyped, the size of the set is claimed to be $q+\frac{q-p}{p-1}$. However, \cite{Geertrui} refers to \cite{LSVdV}, supposedly page 243, as the original source of the construction.}, where $\GF(p)$ is any proper subfield of $\GF(q)$; these are the symmetric differences of particular Rédei type blocking sets and a Rédei line of it.
			\item $2q-k$, where $k<(q+3)/2$ is the number of directions determined by a function $f\colon \GF(q)\to \GF(q)$; this is the symmetric difference of any small Rédei type blocking set (of size $q+k$) and a Rédei line of it \cite{Geertrui}.
		\end{enumerate}
	\end{theorem}
	
	These are the constructions we have found in the literature which work in $\PG(2,q)$ regardless $q$ odd or even. When $q$ is even, there exist a large variety of sets of even type of different sizes in the critical window (let us mention that so-called Korchmáros-Mazzoca arcs are of this kind, but several other constructions can be found, see \cite{KMarcs, SzonyiWeiner}, for example).
	
	For general $q$, we see that in Theorem \ref{UTSconstr}, all examples but the second are the symmetric difference of a small Rédei type blocking set and one of its Rédei lines. We know by Blokhuis' celebrated paper \cite{Blokhuisblsetp} that small blocking sets other than lines exist in $\PG(2,q)$ if and only if $q$ is not a prime. Therefore the second construction is of particular importance as, besides the trivial first construction, it is the only one which works in $\PG(2,p)$, $p$ prime as well.
	
	In the light of the above comments, the following generalisation is quite obvious, yet we could not find it in the literature.
	
	\begin{theorem}\label{smallblsymmdiff}
		Let $\cB_1$ and $\cB_2$ be two small blocking sets in $\PG(2,q)$. Then the symmetric difference $\cB_1\triangle \cB_2$ of $\cB_1$ and $\cB_2$ is a set without tangents.
	\end{theorem}
	
	Instead of proving this result, we shall give a more general construction. Let us first remark that in his seminar paper \cite{Szonyi1modp}, Szőnyi proved that if $\cB$ is any small and minimal blocking set (with respect to containment) in $\PG(2,q)$, where $q=p^h$, $p$ prime, then every line intersects $\cB$ in $1~\pmod p$ points. Therefore the type $T$ of a small blocking set is rather lacunary: if $t_1\neq t_2\in T$, then $|t_1-t_2|\geq p$. This property shows that the following construction is feasible with small blocking sets.
	
	\begin{theorem}\label{gensymmdiff}
		Let $B_i$ be a set of type $T_i$, $i=1,2$. Suppose that $\{t\pm1\colon t\in T_1\}\cap T_2 = \emptyset$. Then $B_1\triangle B_2$ is a set without tangents.
	\end{theorem}
	\begin{proof}
		Let $\ell$ be a line intersecting $B_1$ in $t_1$, $B_2$ in $t_2$, and $B_1\cap B_2$ in $r$ points ($t_i\in T_i$). Then $|\ell\cap (B_1\triangle B_2)|=t_1+t_2-2r$. Without loss of generality we may assume $t_1\leq t_2$. Note that $r\leq t_1$. If $r=t_1$, then $|\ell\cap (B_1\triangle B_2)|=t_2-t_1$, which cannot be $1$ by our assumption. If $r<t_1$, then $|\ell\cap (B_1\triangle B_2)|\geq t_1+t_2-2(t_1-1)\geq 2$.
	\end{proof}
	
	With these generalisations, we obtain new constructions of untouchable sets of size in the critical window but, unfortunately, no new element in the spectrum of their size so far.
	%\todo{Itt lehetne keresni konstrukciókat, pl nem Baer, hanem más (Rédei), vagy lineáris?}
	For example, one can take two small Rédei type blocking sets $\cB_1$ and $\cB_2$ with common Rédei line $\ell$ so that $\cB_1\cap\ell=\cB_2\cap\ell$ and $|(\cB_1\setminus\ell)\cap(\cB_2\setminus\ell)|=1$. Then their symmetric difference is a $1$-avoiding set of size $2q-2$, matching the second result of Theorem \ref{UTSconstr}.

	Two particular problems are the existence of $1$-avoiding sets of size $2q\pm1$ in $\PG(2,q)$. Exhaustive computer search shows that such sets of size $2q-1$ do not exist if $q\leq 11$ is odd or $q=4$, but exist in $\PG(2,8)$ and $\PG(2,16)$. The size $2q+1$ is not covered by Construction \ref{triviconst}, but with computer we have found such examples for $q=4$ and for all $7\leq q\leq 16$. 
	
	\subsubsection{Results for small values of $q$}
	
	For small values of $q$ and for $1\leq k\leq (q+1)/2$, we have performed an exhaustive search with the aid of a computer to determine which values of the critical window are missing from the spectrum. The result is summarized in Table \ref{table:small}. Note that for $q\leq 8$, $\PG(2,q)$ is the unique projective plane of order $q$.

	\newcommand{\emps}{\emptyset}
	\begin{table}[h!]
		\centering
		\begin{tabular}{c||c|c|c|c|c|}
			&  $k=1$             & $k=2$                        & $k=3$   & $k=4$ & $k=5$    \\ \hline\hline
			$q=3$ & $[1,5] \cup \{7\}$            &$\{2\}\cup[5,8]$              &  --     &  --   &  --     \\ \hline 
			$q=4$ &$[1,5]\cup\{7\}$    &$\{2,6,8,10\}$                &  --     &  --   &  --       \\ \hline 
			$q=5$ &$[1,9]\cup{11}$             &$\{2\}\cup[7,11]\cup\{13,14\}$&$[12,19]$&  --   &  --      \\ \hline 
			$q=7$ &$[1,11]\cup\{13\}$  &$\{2\}\cup[10,17]$            &$\{19\}$ &$\emps$&  --      \\ \hline
			$q=8$ &$[1,9]\cup\{11,13\}$&$\{2, 10, 11, 14, 18\}$       & $\emps$ &$\emps$&  --      \\ \hline 
			$q=9$ &$[1,14]\cup\{17\}$  &$\{2,11\}\cup[14,19]$         & $\emps$ &$\emps$&$\emps$   \\ \hline 
		\end{tabular}
		\caption{\label{table:small}For given $q$ and $1\leq k\leq \lfloor(q+1)/2\rfloor$, the table lists those values $m$ for which a $k$-avoiding set of size $m$ \emph{does not exist} in $\PG(2,q)$. The results are based on computer search.}
		\label{tab:smallq}
	\end{table}
	
	Let us remark that $\Sp(k,\Pi_q)$ indeed depends on the structure of the plane: in $\PG(2,9)$, there is no $1$-avoiding set of size $14$, but there is one in the Hughes plane of order $9$ \cite{Blokhuis}.

	%Sidenote: Brouwer and Dover made sets with  secants of only $3$ sizes \cite{Brouwer, Dover}
	
	%Some ideas.
	
	%Following the work of Vandendriessche \cite{Vanden} it is well-believed that for $q$ even, $\PG(2,q)$ contains Korchmáros-Mazzocca $(q+t,t)$ arc for each feasible $t$, i.e., for every even divisor of $q$. It is well known that such a structure contains $q/t+1$ concurrent secants, each of size  $t$. So one might build mighty sets as follows. To build a set of size $s$ which avoids an $m$-secant, where $s$ is in the nontrivial window $(m-1)q<s<(m+1)q$, take the binary representation of $s-(m-1)q$ and pick the respective disjoint copies of KM-arcs on the same nucleus.
	%If the support size is $r$, the intersection size will be an even number which is a $t$-secant of a distinct KM-arc, or at most $2r$. The number of lines needed, however, is depending on the local values. 
	%$\bullet$ One might make the symmetric difference of the secants of distinct KM arcs and see what happens...

	\section{Constructions of size $m$ in the critical window}\label{sec:4}
	
	Next we give several constructions based on specific substructures of projective planes. The common idea behind them is that we find a $k$-avoiding set of size in the critical window with very few line intersection sizes (so the type of the set is very restricted), and then we modify it either by adding large sets of collinear points to it (lines and subsets of lines, similar to Construction \ref{triviconst}), or by removing some points from it. We do this so that we can keep track of the intersection sizes with lines. The goal is to increase or decrease $k$ and to let the size vary so that we get a large interval in the spectrum. 
	
	\subsection{Constructions for $q$ even via maximal arcs}
	
	The first substructures we utilize are maximal $(k,n)$-arcs.
	
	%\todo{Alt. szerkezet: constr; prop: const tuli; thm: kovetkezmeny a k-avoidingra vonatkozóan}
	
	\begin{defi}[$(k,n)$-arcs, maximal $(k,n)$-arcs]
		A set $K$ of $k$ points is called a  \emph{$(k, n)$-arc} if 
		\[\max_{\ell\in \cL} \{|\ell \cap K| : \ell\in \mathcal{L}(\Pi_q) \} = n.\] For given $q$ and $n$, the cardinality of a $(k, n)$-arc $K$ can never exceed $(n- 1)(q + 1) + 1$. A \emph{maximal $(k, n)$-arc} is a $(k,n)$-arc of size $k=(n-1)(q+1)+1$. Equivalently, a maximal $(k,n)$-arc can be defined as a nonempty point set of the projective plane which meets every line in either $n$ or in zero points.
	\end{defi}
	
	While in $\PG(2,q)$ with $q$ odd, Ball,  Blokhuis and Mazzocca proved that no non-trivial maximal arcs exist \cite{BallBlokhuis}, we do have examples when $q$ is even. From a classical result of Denniston \cite{Denniston}, we know that maximal $(k,n)$-arcs exist in $\PG(2,q)$, $q=2^h$, for every divisor $n$ of $q$. In the upcoming construction, we use maximal $(3q+4,4)$-arcs to fill the upper half of the critical window for $3\leq k\leq q/4+3$.
	
	\iffalse %kell ez?
	For sake of convenience, let us formulate a lemma.
	
	\begin{lemma}\label{windowlem}
		Let $\cA$ be a set of type $T$. Let $P\notin\cA$, let $\ell_1,\ldots,\ell_r$ be lines through $P$ which intersect $\cA$ in $a$ points, and let $\ell^1$ and $\ell^2$ be two further lines which intersect $\cA$ in $a_1$ and $a_2$ points, respectively. Let $X_i\subset\ell^i\setminus(\cA\cup\{P\})$, $i=1,2$, with $0\leq|X_i|\leq q-a_i$. Consider 
		\[\cS=\cA \cup \bigcup_{i=1}^{r}(\ell_i\setminus\{P\}) \cup X_1\cup X_2.\]
		Then $\cS$ is of size $|\cA|+r(q-a)+|X_1|+|X_2|$, and $\cS$ is of type $T\cup \{|X_i|+a_i\}_{i=1}^2\cup \{q\} \cup $
	\end{lemma}
	\fi
	
	\begin{const}\label{evenq/4}
		Suppose that $q$ is even. Let $\cA$ be a maximal $(3q+4,4)$-arc in $\Pi_q$, and choose a point $P\notin \cA$. Let $\ell_1,\ldots,\ell_{q/4}$ be the skew lines to $\cA$ through $P$, and let $0\leq r\leq q/4-1$ (Option 1) or $0\leq r\leq q/4$ (Option 2). Let $\ell'=\ell_{r+1}$ be a skew line through $P$ (Option 1), or let $\ell'$ be a $4$-secant to $\cA$ through $P$ (Option 2). Let $X\subset \ell'\setminus(\{P\}\cup\cA)$ be a set of $t$ points, $0\leq t\leq q-4$. Let 
		\[\cS=\cA\cup \bigcup_{i=1}^r(\ell_i\setminus\{P\}) \cup X.\]
	\end{const}
	
	\begin{prop}\label{evenq/4prop}
		The set $\cS$ in Construction \ref{evenq/4} has $(r+3)q+4+t$ points, and for each line $\ell$, $|\ell\cap \cS|\in\{0,4,t,q,r,r+1,r+4,r+5\}$ in case of Option 1, and $|\ell\cap \cS|\in\{0,4,t+4,q,r,r+1,r+4,r+5\}$ in case of Option 2. In particular, if $r\geq 2$, then at least one of the Options for $\cS$ yields an $(r+3)$-avoiding set.
	\end{prop}
	\begin{proof}
		Let $\ell$ be any line. If $P\in\ell$, then $\ell$ is either a skew or a $4$-secant line to $\cA$, hence $|\ell\cap\cS|\in\{0,4,q,t\}$ in case of Option 1, and $|\ell\cap\cS|\in\{0,4,q,t+4\}$ in case of Option 2. If $P\notin\ell$, then $\ell$ meets $\cA$ in $0$ or $4$ points, $\ell$ contains one point of $\cS$ on each line $\ell_1,\ldots,\ell_r$ and at most one more point of $\cS$ on $\ell'$. Thus $|\ell\cap\cS|\in\{r,r+1,r+4,r+5\}$. By $2\leq r\leq q/4$, we have $q\geq 8$, and thus $4$-secants and $q$-secants are not $(r+3)$-secants.
	\end{proof}
	
	\begin{theorem}\label{evenq/4thm}
		Let $q\geq 4$ be even, $3\leq k\leq q/4+3$, and let $m\in[kq+4,(k+1)q]$. Then there exists a $k$-avoiding set of size $m$ in $\PG(2,q)$ unless $k=4$ and $m=4q+4$, or $k=q/4+3$ and $m=kq+k$.
	\end{theorem}
	\begin{proof}
		We want to apply Construction \ref{evenq/4} with $r=k-3$. 
		If $5\leq k\leq q/4+2$, then $2\leq r\leq q/4-1$, so by Proposition \ref{evenq/4prop}, at least one of the two options for Construction \ref{evenq/4} yields a $k$-avoiding set of size $kq+4+t$ for any $0\leq t\leq q-4$. Note that this also works for $k=3$.
		
		If $k=4$ and $q\geq 8$, then consider the set $\cS$ from Construction \ref{evenq/4} with $r=1$ (here $r\leq q/4-1$, so we may use both options), and let $\cS'=\cS\cup\{P\}$. Then lines through $P$ intersect $\cS'$ in $1,5,q+1$ and $t$ or $t+4$ points (with regard to the two Options), other lines intersect $\cS'$ in $1,2,5$ or $6$ points. As $|\cS'|=4q+5+t$, we get the stated result. If $k=4$ and $q=4$, then $kq+4=(k+1)q=4q+4$, thus by the exception, we have nothing to prove.
		
		If $k=q/4+3$, then $r=q/4$, thus we have only Option 2 in Construction \ref{evenq/4}. This gives us a $k$-avoiding set of size $kq+4+t$ for each $0\leq t\leq q-4$, $t\neq k-4$.    
	\end{proof}

	We can get much closer to close the critical window provided that there exist two disjoint maximal $(3q+4,4)$-arcs. In $\PG(2,16)$, this is  indeed true \cite{Hami}, and we conjecture that it is true in general when $q\geq 16$ is even.%\todo{van-e: Singer, ált PGL ellenőrzés?}
	
	\begin{conj}\label{2maxarcconj}
		For every even $q\geq 16$, there exist two disjoint $(3q+4,4)$-arcs in $\PG(2,q)$.
	\end{conj}
	
	\begin{lemma}\label{2maxarclemma}
		Suppose that there are two disjoint $(3q+4,4)$-arcs in $\Pi_q$, $\cK_1$ and $\cK_2$, and let $\cA=\cK_1\cup\cK_2$. Then there exists a point not in $\cA$ on which there are more than $q/16 - 1$ skew lines to $\cA$.
	\end{lemma}
	\begin{proof}
		Take an arbitrary point $P\in\cA$. Without loss of generality, we may assume that $P\in\cK_1$. Then every line through $P$ are $4$-secants to $\cK_1$, and $(3q+4)/4=\frac34q+1$ lines through $P$ are $4$-secants to $\cK_2$. Hence the number of $4$-secants and $8$-secants to $\cA$ through any point of $\cA$ is $\frac{q}{4}$ and $\frac34q+1$, respectively. Thus the number of skew lines to $\cA$ is  $q^2+q+1-\left(\frac14|\cA|\cdot\frac{q}{4} + \frac18|\cA|\cdot\left(\frac34q+1\right)\right)=\frac{1}{16}q^2-q$. (Note that this must be non-negative, hence $q\geq 16$ follows.) Thus the average number of skew lines on a point not in $\cA$ is \[\frac{\left(\frac{1}{16}q^2-q\right)(q+1)}{q^2+q+1-|\cA|} = \frac{q}{16}-\frac58 -\frac{\left(\frac{59q}{16}+\frac{35}{8}\right)}{q^2-5q-7}.\]
		This quantity is $0$ for $q=16$ and strictly larger than $q/16-1$ if $q>16$.
	\end{proof}

	\begin{const}\label{2maxarcconst}
		Suppose that in $\Pi_q$, there are two disjoint $(3q+4,4)$-arcs, and let $\cA$ be their union. Choose a point $P\notin \cA$, let $w$ denote the number of skew lines to $\cA$ through $P$, and denote these lines by $\ell_1,\ldots,\ell_w$. Let $0\leq r\leq w-1$. Let $\ell$ be another skew line through $P$ (Option 1), or let $\ell$ be a $4$-secant or $8$-secant to $\cA$ through $P$ (Option 2). Let $X\subset \ell\setminus(\{P\}\cup\cA)$ be a set of $t$ points, where $0\leq t\leq q$ in case of Option 1, and $0\leq t\leq q-8$ in case of Option 2. Let 
		\[\cS=\cA\cup \bigcup_{i=1}^r(\ell_i\setminus\{P\}) \cup X.\]
	\end{const}
	
	\begin{prop}\label{2maxarcprop}
		The set $\cS$ in Construction \ref{2maxarcconst} has $(r+6)q+8+t$ points, and for each line $\ell$, $|\ell\cap \cS|\in\{0,4,8, t,q,r,r+1,r+4,r+5,r+8,r+9\}$ in case of Option 1, and $|\ell\cap \cS|\in\{0,4,8, t+\varepsilon 4,q,r,r+1,r+4,r+5,r+8,r+9\}$, $\varepsilon\in\{1,2\}$ in case of Option 2. In particular, if $1\neq r\geq q-7$, then at least one of the Options for $\cS$ yields an $(r+7)$-avoiding set.
	\end{prop}
	\begin{proof}
		The size of $\cS$ is self evident. Let $\ell$ be any line. If $P\in\ell$, then $|\ell\cap\cS|\in\{0,4,8,q,t\}$ (Option 1) or $|\ell\cap\cS|\in\{0,4,8,q,t+\varepsilon 4\}$ (Option 2, $\varepsilon\in\{1,2\}$). If $P\notin\ell$, then $\ell$ intersects $\ell_1\ldots,\ell_r$ in a point of $\cS\setminus\cA$, it may or may not contain a point of $\cS$ on $\ell$, and further $0$, $4$ or $8$ points from $\cA$; hence $|\ell\cap\cS|\in\{r,r+1,r+4,r+5,r+8,r+9\}$. 
	\end{proof}

	\begin{theorem}\label{2maxarcthm}
		If there are two disjoint $(3q+4,4)$-arcs in $\PG(2,q)$, $q\geq 32$, then there exists a $k$-avoiding set in $\PG(2,q)$ of size $m$ for all $7\leq k\leq q/16+6$, $k\neq 8$, and $m\in[(k-1)q+8,kq+8]$.
	\end{theorem}
	\begin{proof}
		Let $\cA$ be the union of the two disjoint $(3q+4,4)$-arcs. By Lemma \ref{2maxarclemma}, there is a point $P\notin\cA$ with at least $q/16$ skew lines to $\cA$. Let $r=k-7$, and apply Construction \ref{2maxarcconst} to obtain the set $\cS$. If $r\neq 2$, by Proposition \ref{2maxarcprop}, we obtain a $k$-avoiding set of size $(k-1)q+8+t$ for all $0\leq t\leq q$ (note that for $t\geq q-7$, we may apply Option 1 as $t > k$ follows from $k\leq q/16+6$ and $q\geq 32$).
	\end{proof}
	
	Combining Theorems \ref{evenq/4thm} and \ref{2maxarcthm}, we get the following result, which almost completely closes the critical window.
	
	\begin{cor}\label{2maxarccor}
		If Conjecture \ref{2maxarcconj} holds and $q\geq 32$, then for every $7\leq k\leq q/16+6$, $k\neq 8$, and $m\in[(k-1)q+8,(k+1)q]$, there exists a $k$-avoiding set of size $m$ in $\PG(2,q)$. 
	\end{cor}
	\begin{comment}

	\begin{question} What is the possible size of the intersection of two (or more) complete arcs? (Supposing that they share or not share the same nucleus)
	\end{question}
	
	\subsection{Construction via lines, $q$ even case}
	
	Summary, to be completed.
	Take a line with its point set $\{P_1, \ldots P_{q+1} \}$. On each point $P_i$, chose $t_i$ incident lines independently, uniformly at random.  The numbers $t_i$ may differ by at most one from each other, and $\sum_i t_i = T$.
	
	\begin{itemize}
	\item determine the size of the size of the union of the lines. This is approximately $ q^2(1-e^-c)$ where $c=T/(q+1)$.
	\item The the number of points incident with an even number of chosen lines. The intersection number follows Poisson distrib. ($~q^2c^m/m!\cdot e^{-c} $ for $m$-wise intersection points)
	
	\item erasing all points on even number of lines given a set of even type or its complement, dependening on the parity of $T$.
	
	\end{itemize}

	\end{comment}
	
	\subsection{Construction for $q$ square via Baer subplanes}
	
	The following construction relies on Baer subplanes; that is, subplanes of order $\sqrt{q}$. A Baer subplane has $q+\sqrt{q}+1$ points, and each line intersects it in either one or $\sqrt{q}+1$ points. It is well-known that the point set of $\PG(2,q)$, $q$ square, can be partitioned into (the point sets of) $q-\sqrt{q}+1$ disjoint Baer subplanes. If we take $s$ pairwise disjoint Baer subplanes, then for any line $\ell$, either $\ell$ intersects each of the $s$ Baer subplanes in exactly one point, or $\ell$ intersects precisely one of them in $\sqrt{q}+1$ points and each one of the rest in exactly one point. Hence the union of $s$ disjoint Baer subplanes is a set of type $(s,\sqrt{q}+s)$ of size $s(q+\sqrt{q}+1)$. Through any point $P$ not in the union, there are exactly $s$ distinct $(\sqrt{q}+s)$-secants, the remaining $q+1-s$ lines are $s$-secants.
	
	%\subsubsection{First construction: Baer subplanes and lines}
	
	\begin{const}\label{Baer1const}
		Let $q$ be a square prime power.
		Take the point set $\mathcal{A}$ of $s$ disjoint Baer subplanes of $\PG(2,q)$ ($1\leq s\leq q-\sqrt{q}$) and a point $P$ outside the Baer subplanes. Choose $r$ lines $\ell_1, \ell_2, \ldots \ell_{r}$ through $P$ so that each of them intersects $\cA$ in $s$ points, and let $\ell_{r+1}$ be a line through $P$ intersecting $\cA$ in either $s$ points (Option~1) or $\sqrt{q}+s$ points (Option 2), where $0\leq r\leq q-s$.\\ 
		Let $H$ be a $t$-set from $\ell_{r+1}\setminus \mathcal{A}$ ($0\leq t\leq q-s$ in Option 1 and $0\leq t\leq q-\sqrt{q}-s$ in Option 2) and let 
		\[\cS=\mathcal{A}\cup \bigcup_{i=1}^r (\ell_i\setminus\{P\}) \cup H.\]
	\end{const}
	
	\begin{prop}\label{Baer1prop}
		The set $\cS$ in Construction \ref{Baer1const} has $s(q+\sqrt{q}+1)+r(q-s)+t$ points, and each line different from $\ell_{r+1}$ intersects $\cS$ in at most $s+r+1$ or at least $s+\sqrt{q}$ points, and $\ell_{r+1}$ is an $(s+t)$-secant (Option 1) or an $(s+\sqrt{q}+t)$-secant (Option 2).
	\end{prop}
	\begin{proof}
		The size of $\cS$ is self evident. Let $\ell$ be a line. Suppose first that $\ell$ intersects one of the Baer subplanes in $\cS$ in $\sqrt{q}+1$ ponits. Then, as $\ell$ intersects each of the other $s-1$ Baer subplanes in $1$ point, $|\ell\cap\cS|\geq s+\sqrt{q}$. Suppose now that $\ell$ intersects every Baer subplane of $\cS$ in $1$ point.
		If $P\notin\ell$, then $\ell$ contains at most further $r+1$ points of $\cS$ from the lines $\ell_1,\ldots,\ell_{r+1}$, and thus $|\ell\cap \cS|\leq s+r+1$. If $P\in\ell$, then $|\ell\cap\cS|\in\{q,s+t,s\}$ (Option 1) or $|\ell\cap\cS|\in\{q,s+\sqrt{q}+t,s\}$ (Option 2).
	\end{proof}
	
	Let us formulate the direct consequence of the above construction in a lemma.
	
	\begin{lemma}\label{lemmaBaer}
		Suppose that $q$ is a square prime power. Let $2\leq k\leq q-1$, and let $s,r,t$ be integers such that $1\leq s\leq q-\sqrt{q}$, $0\leq r\leq q-s$, $s+r+2\leq k\leq s+\sqrt{q}-1$, and $0\leq t\leq q-s$, $t\neq k-s$. Then there exists a $k$-avoiding set of points in $\PG(2,q)$ of size $(k-1)q+\delta_k(s,r,t)$, where
		\[\delta_k(s,r,t)=s\sqrt{q}-(k-s-r-1)q-s(r-1)+t.\]
		Moreover, if $k\leq q-\sqrt{q}$, then the condition $t\neq k-s$ may be dropped.
	\end{lemma}
	\begin{proof}
		Apply Option 1 of Construction \ref{Baer1const} with the parameters $s,r$ and $t$ chosen as in the assumptions of the present Lemma  ($t\neq k-s$) to obtain the point set $\cS$. Because of the choice of $s$, $r$ and $t$, $\cS$ avoids $k$-secants. If $t=k-s$ and $k\leq q-\sqrt{q}$, then $t\leq q-\sqrt{q}-s$ holds, thus we may choose Option 2 to obtain a set of the same size avoiding $k$-secants. Furthermore, $|\cS|-(k-1)q=s(q+\sqrt{q}+1)+r(q-s)+t -(k-1)q = s\sqrt{q} - (k-1-s-r)q-s(r-1)+t$, as stated.
	\end{proof}
	
	\begin{theorem}\label{baer1thm}
		Let $q$ be a square, $3\leq k\leq q-1$, $m\in[(k-1)q+2,(k-1)q+(k-2)\sqrt{q}]$. Then there exists a $k$-avoiding set in $\PG(2,q)$ of size $m$.
	\end{theorem}
	\begin{proof}
		As usual, we may assume that $k\leq (q+1)/2$. Note that $k\leq q-\sqrt{q}$ follows. By Lemma \ref{lemmaBaer}, it suffices to show that the values $\delta_k(s,r,t)$ can take all integer values of the interval $[2,(k-2)\sqrt{q}]$. Note that $\delta_k(s,r,t)$ is increasing in $s$, $r$ and $t$ as well.
		
		Let $I_k[s,r]$ denote the interval $[\delta_k(s,r,0),\delta_k(s,r,q-s)]$. If $r$ and $s$ are appropriate for Lemma \ref{lemmaBaer}, then $\delta_k(r,s,t)$ can take each value of $I_k[s,r]$. Thus it is enough to show that the intervals $I_k[s,r]$ cover the interval $[2,(k-2)\sqrt{q}]$. Let us restrict the values of $r$ to be used to $0\leq r\leq 2$.
		
		The parameters $s_0=\max\{1,k-\sqrt{q}+1\}$, $r_0=t_0=0$ and $s_1=k-2$, $r_1=0$, $t_1 = q-(k-2)$ are suitable for Lemma \ref{lemmaBaer}. Suppose first $k\geq \sqrt{q}$. Then we see that $\delta_k(k-\sqrt{q}+1,0,0) = k(\sqrt{q}+1)-q(\sqrt{q}-1)+1\leq 2$ by $k\leq(q+1)/2$. If $3\leq k\leq \sqrt{q-1}$, then $\delta_k(1,0,0) = \sqrt{q}-(k-2)q +1\leq 2$ also holds. On the other hand, we also see that $\delta_k(k-2,0,q-k+2)=(k-2)\sqrt{q}$. Thus both $2$ and $(k-2)\sqrt{q}$ are in the values covered by $\delta_k(r,s,t)$.
		
		Now observe that $\delta_k(s,r,q-s)=\delta_k(s,r+1,0)$ (thus $\cup_{r=0}^2 I_k[s,r]$ is an interval), furthermore 
		\begin{equation}
			\begin{split}
				\delta_k(s,2,q-s) &= s\sqrt{q}-(k-s-3)q-s+q-s =s\sqrt{q} - (k-s-2)q + 2q - 2s \\ &\geq  s\sqrt{q}+\sqrt{q}-(k-s-2)q + s + 1 =\delta_k(s+1,0,0),  
			\end{split}
		\end{equation}
		where we use that $2q-2s\geq\sqrt{q}+s+1$, which follows from $s\leq k-2 < q/2$. Thus indeed, \[\bigcup_{s=k-\sqrt{q}+1}^{k-2}\bigcup_{r=0}^2 I[s,r]\] covers $[2,(k-2)\sqrt{q}]$.  
	\end{proof}
	
	As a corollary, we obtain easily the following
	
	\begin{theorem}\label{BaerCor}
		Let $q\geq 25$ be a square prime power, $2\sqrt{q}+2\leq k \leq q-2\sqrt{q}-1$. Then there exists an $m$-set of points avoiding $k$-secants for all $m$ with $0\leq m\leq q^2+q+1$.
	\end{theorem}
	\begin{proof}
		Due to Proposition \ref{trivikonst}, it is enough to verify that for $(k-1)q+2\leq m \leq (k+1)q-1$, there is an $m$-set avoiding $k$-secants. By Lemma \ref{lemmaBaer}, this holds if $(k-1)q+2\leq m \leq (k-1)q+(k-2)\sqrt{q}$. As $k\geq 2\sqrt{q}+2$, $(k-2)\sqrt{q}\geq 2q-1$, and thus the statement follows.    
	\end{proof}

	%\begin{prop}
	%   Let $q\in \{4,9,16\}$. Then ...\todo{géppel megnézni}
	%\end{prop}
	%\todo[inline]{kombinálható-e az ívekkel, q páros esetben: vajon a kritikus ablak felső felét oldja-e meg ez is ha k 2gyok(q) alá megy?}
	\begin{comment}
	
	\begin{theorem}
	Construction \ref{Baer} gives rise to a $K$-set for $$K=s(q+\sqrt{q}+1)+r(q-s)+t=q(s+\frac{s}{\sqrt{q}}+r)+s(1-r)+t$$ points, which avoids $s+r+\frac{s}{\sqrt{q}}$-secants and $s+r+\frac{s}{\sqrt{q}}+1$-secants, provided that $q-\sqrt{q}>s\ge \sqrt{q}$ and $r...$
	\end{theorem}
	
	\begin{cor}
	If $q$ is a square,    $k$-secants are avoidable for each set size $m$, provided that $q-\sqrt{q}>k>\sqrt{q}$
	\end{cor}
	
	\end{comment}
	
	\begin{remark}
		Taking an appropriate subset of the union of some disjoint Baer subplanes, we can constuct further $k$-avoiding sets. With the help of these, the result of Theorem \ref{BaerCor} can be extended to $2\sqrt{q}\leq k\leq q-2\sqrt{q}+1$.
	\end{remark}

	\subsection{Construction with lines}
	
	The next construction may be regarded as a generalisation of Construction \ref{triviconst}: we take the union of the point sets of some lines, but here the lines need not to be concurrent; then we may remove some points from the points and, under particular conditions, we can also add some sets of collinear set to let the size vary. For example, the upcoming Construction \ref{linesconst} includes three lines in general position with their intersection points removed ($r=3$, $s=0$,  $H=\emptyset$), known as the vertexless triangle, which is a $2$-avoiding set of size $3(q-1)$.
	
	\begin{const}\label{linesconst}
		Take a set $\cL$ of $2\leq r< q/2-1$ lines in $\Pi_q$, and let $0\leq s\leq r-2$ arbitrary. Let $\cP_i$ be the set of points covered by exactly $i$ lines of $\cL$, and let $\cP=\cup_{i=1}^r\cP_i$ be the set of points covered by the lines in $\cL$. Suppose that $\cP_i=\emptyset$ for all $2\leq i\leq s+1$ (note that this condition requires nothing if $s=0$), and let $H\subset \cup_{i=s+3}^{r}$ be arbitrary. For every $1\leq i\leq s$, let $\ell_i\notin\cL$ and let $X_i\subset(\cP_0\cap\ell_i)$. Assume that $X_1,\ldots, X_s$ are pairwise disjoint, and let $X=\cup_{i=1}^s X_i$. Let 
		\[\cS=\left(\cP\setminus(\cP_{s+2}\cup H)\right) \cup X.\]
	\end{const}
	
	\begin{prop}\label{linesprop}
		Consider Construction \ref{linesconst}, and let $\ell$ be an arbitrary line. Then 
		\begin{equation}\label{ellcount}
			|\ell\cap\cP|= r - \sum_{i=1}^r (i-1)|\ell\cap\cP_i|.   
		\end{equation}
		If $\ell\notin\{\ell_1,\ldots,\ell_s\}$, then $|\cS\cap\ell|\neq r-1$. 
	\end{prop}
	\begin{proof}
		Clearly, $|\ell\cap\cP|=\sum_{i=1}^r|\ell\cap\cP_i|$. As $\ell$ meets each line of $\cL$ in a unique point, it is also clear that $\sum_{i=1}^r i|\ell\cap\cP_i|=r$. Subtracting these gives \eqref{ellcount}. Assume now $\ell\notin \{\ell_1,\ldots,\ell_s\}$; this yields $|\ell\cap X|\leq s$. If $\ell$ meets every line of $\cL$ in a point of $\cP_1$, then $|\ell\cap \cP|=\left|\ell\cap\left(\cP\setminus(\cP_2\cup H)\right)\right|=r$, hence $|\ell \cap \cS|\geq r$. Suppose now that $\ell$ contains a point $P$ from $\cP$ which is covered by $j\geq 2$ lines of $\cL$. By \eqref{ellcount}, we have $|\ell\cap\cP|\leq r-(j-1)$. By the assumptions of Construction \ref{linesconst}, $j\geq s+2$. Assume first $j\geq s+3$. Then $|\ell\cap\cS|\leq |\ell\cap\cP|+|\ell\cap X|\leq r-2$. If $j=s+2$, then $P\notin\cS$, and thus $|\ell\cap\cS|\leq |\ell\cap\cP|-1 +|\ell\cap X|\leq r-2$.
	\end{proof}
	
	Note that by Proposition \ref{linesprop}, by taking the union of a set $\cL$ of any $k+1$ lines and remove some of the points covered by three or more lines, we always obtain a $k$-avoiding set. Moreover, if there are no points covered by exactly two lines of $\cL$, we can add a suitable collinear subset of points to provide a large interval in the spectrum. Note that such a line set is the dual of a $2$-avoiding set.
	
	%\todo{FIN REM-be}
	%\todo{$$(0,1,3)$$-halmazok vagy majdnem?}
	
	In the proof of the next theorem, we use homogeneous coordinates $(x:y:z)$ and $[m:w:b]$ to represent the points and the lines of $\PG(2,q)$, respectively, where a point $(x:y:z)$ is on the line $[m:w:b]$ if and only if $mx+wy+zb=0$. 
	
	\begin{theorem}\label{linesthm1}
		Assume that $d\mid q-1$, $2\leq d< q-1$, and let $k=3d-1$. Then there exists a $k$-avoiding set in $\PG(2,q)$ of size $m$, provided that
		\begin{compactitem}
			\item $d=2$ ($k=5$) and $m\in[6q-12,6q-8]$, or
			\item $d=3$ ($k=8$), $q\geq 11$ and $m\in [9q-27,10q-31]$, or
			\item $d\geq 4$ and $m\in[3dq-3d^2,3dq-3d^2+q-d+2]$.
		\end{compactitem}
	\end{theorem}
	\begin{proof}
		We apply Construction \ref{linesconst}. Let $G$ be a subgroup of order $d$ of the multiplicative group $\GF(q)^*$ of $\GF(q)$. Let $\cL_1=\{[m:-1:0]\colon m\in G\}$, $\cL_2= \{[-1:0:a]\colon a\in G\}$, $\cL_3= \{[0:-1:b]\colon b\in G\}$, and $\cL=\cL_1\cup\cL_2\cup\cL_3$. Thus $\cL$ consists of $r=3d$ lines passing through one of the points of $V=\{(0:0:1)$, $(0:1:0)$, $(1:0:0)\}$, and $V\subset \cP_d$. Note that if $1\leq i<j\leq 3$, then $\{\ell\cap\ell'\colon \ell\in\cL_i, \ell'\in\cL_j\}=\{(x:y:1)\colon x,y\in G\}$, which set we denote by $Z$. It follows that $\cP_2\setminus V=\emptyset$, $\cP_3\setminus V=Z$ (note that if $d=2$ or $3$, the three points in $V$ are somewhat special), and $\cP_i=\empty$ if $d\neq i\geq 4$. Then $|\cP|=3+3d(q-d)+d^2 = 3dq - 2d^2 + 3$. Note that $|[0:1:0]\cap \cP_0|=q-1-d$ and, if $m\notin G$, then $|[m:-1:0]\cap\cP_0|=q-2d$.
		
		If $d=2$, then $|\cP_2|=|V|=3$, $|\cP_3|=d^2$. Let $s=0$ and $H\subset\cP_3$ arbitrary. By Proposition \ref{linesprop}, in this way we can construct a $5$-avoiding set of size $m$ for any $m\in[6q-12,6q-8]$.
		
		If $d = 3$, then $\cP_2=\emptyset$ and $\cP_3=V\cup Z$. Let $s=1$, $H=\emptyset$, and let $X\subset[0:1:0]\cap \cP_0$ be arbitrary. Then $|X\cap \cS|=X+d$. By Proposition \ref{linesprop}, $|\cS|=9q-27+|X|$, and $\cS$ is $(3d-1)=8$-avoiding unless $|X|=5$. In this case, we may choose $X$ to be a subset of $\cP_0$ on a $(q-2d)=(q-6)$-secant line to $\cP_0$ (here we need $q\geq 11$).
		
		If $d\geq 4$, then $\cP_2=\emptyset$, $\cP_3=Z$, $\cP_d=V$. Let $H\subset V$ ($0\leq |H|\leq 3$) and $X\subset[0:1:0]\cap \cP_0$ ($0\leq |X|\leq q-1-d$) be arbitrary. By Proposition \ref{linesprop}, we can construct a $(3d-1)$-avoiding set of any size in $[3dq-3d^2,3dq-3d^2+q-d+2]$.
	\end{proof}
	
	Note that if $d$ is between (roughly) $\sqrt{q/3}$ and $\sqrt{2q/3}$, then the whole interval of length $q-d+3$ covered by Theorem \ref{linesthm1} is in the critical window. Let us also remark that Theorem \ref{linesthm1}, unlike our other constructions, works for prime order planes as well. 
	%For example, for $q=457$ (prime), $d=12$ is a divisor of $q-1$, 
	
	\begin{remark}
		Instead of using the whole line set $\cL_3$ in Theorem \ref{linesthm1}, if we take only a subset of size $u$ of $\cL_3$, we can obtain $(2d+u-1)$-avoiding sets of each size in $[(2d+u)q-2d^2-ud,(2d+u)q-2d^2]$. Here we assume $u,d\geq 4$; small $d$ or $u$ requires a bit of caution.
	\end{remark}
	
	\iffalse

	\begin{theorem}
		We can do something similar with a subgroup of $G$ and subsets of it, so we get $k$-avoiding sets for more values of $k$, but the length of the interval of possible sizes inside the critical window is considerably shorter, because we cannot extend the construction with a collinear set (only $s=0$ works).
	\end{theorem}
	\begin{proof}
		Similarly as in the previous proof, $G$ is a subgroup of $\GF(q)^*$ order $d$, $A,B\subset G$, and we take the $r=d+|A|+|B|$ lines through $(0:0:1)$, $(0:1:0)$ and $(1:0:0)$ with parameters in $G$, $A$ and $B$, respectively. Then $\cP_3\setminus V = Z=\{(a:b:1)\colon a\in A, b\in B\}$, $\cP_2\setminus V = \{(a:b:1)\colon a\in A, b\in G\setminus B\}\cup \{(a:b:1)\colon a\in G\setminus A, b\in B\}$, $|\cP| = 3 + dq + |A|(q-d)+|B|(q-d) = (d+|A|+|B|)q - (|A|+|B|)d + 3$. We may take $s=0$ and $H\subset \cP_3$ arbitrarily to obtain $(r-1)$-avoiding sets of size in an interval of length $|\cP_3|+1=|A|\cdot|B|+1$.
	\end{proof}
	
	\fi
	
	Next we take a subplane and the dual of a $2$-avoiding set in it to obtain an interval of sizes in the spectrum. 
	
	\begin{theorem}\label{linesthm2}
		Suppose that $\Pi_p$, $p\geq 5$, is a subplane of $\Pi_q$, and let $3p+6\leq k \leq p^2+p-1$. Then there exists a $k$-avoiding set in $\Pi_q$ of size $m$ for all $m\in [(k+1)(q-p), (k+2)(q-p)]$.
	\end{theorem}
	\begin{proof}
		We apply Construction \ref{linesconst} and Proposition \ref{linesprop} and the notation therein.
		Let $\Pi_p\simeq\PG(2,p)$ be a subplane of $\PG(2,q)$, and let $\ell$ be a line of $\PG(2,q)$ with $\ell\cap\Pi_p=\{P_0,\ldots, P_p\}$. Take all lines of $\Pi_p$ through $P_0$, $P_1$ and $P_2$ and $k+1-3p$ more lines of $\Pi_p$ different from $\ell$ so that each $P_i$, $3\leq i\leq p$, is covered by either $0$ or at least $3$ of them. (This can be done by $p\geq 5$.) This set $\cL$ of $k+1$ lines, considered as lines of $\PG(2,q)$, cover each point of $\Pi_p\setminus\ell$ at least three times, the points of $\ell$ either $0$ or at least three times, and all other points at most once. Thus $\cP_2=\emptyset$ and $|\cP_1|=|\cP\setminus \cup_{i=3}^{k+1}\cP_i|=(k+1)(q-p)$. As $|\ell\cap\cP_0|\geq q-p$, we may apply Proposition \ref{linesprop} with $s=1$, $\ell_1=\ell$ to finish the proof. 
	\end{proof}
	
	The most interesting case of the above result is when $q=p^3$ and $k\approx p^2$, as in this case the interval covered by Theorem \ref{linesthm2} is inside the critical window and its length is linear in $q$.
	
	\begin{cor}
		If $q=p^3$, $p\geq 5$ a prime power, and $k=p^2+c$ ($-2\leq c\leq p$), then 
		\[[kq-(c+1)p, (k+1)q-(c+2)p]\subset\Sp(k,q).\]
	\end{cor}
	
	\iffalse
	
	Note that since $k+1<q/2$, we drop at most $k$ points from each line.
	
	If $k>\sqrt{2q}$ than we might actually drop as many as $q$ pts. If $k>2\sqrt{q}$ then we might actually drop more than $2q$ points, which means we might obtain every possible size $m$ which is not handled in the trivial construction.
	
	\fi

	\section{Concluding remarks and open questions.}
	
	In this section we pose some open problems and point out some further related results and questions. Our main conjecture is 
	
	\begin{conj}\label{mainconj} Suppose that $ c\sqrt{q}<k<q+1-c\sqrt{q}$ for some absolute constant $c$. Then $ \Sp(k,q)=[0,q^2+q+1]$, i.e., there exists a $k$-avoiding set in $\PG(2,q)$ of every admissible size.
	\end{conj}
	Note that Theorem \ref{baeres} shows that this holds if $q$ is a square. Moreover  Corollary \ref{2maxarccor} shows that the spectrum $\Sp(k,q)$ consists of all but at most $6$  admissible values  if $q\geq 32$ is even, provided that $7\leq k\leq q/16+6$.  \\
	In fact, we believe that Conjecture
	\ref{mainconj} might  be true  for all integer values of $k\in [c\log(q)), q+1-c\log(q)]$. Recall that some lower bounds on $k$ are necessary to have a complete spectrum, in view of results on untouchable sets and our exhaustive search data. 
	
	It is also natural to require an avoiding property for more than one value of $k$. We may, for example, require that $S$ is $k$-avoiding for each $k$ in a large interval $[a,b]$. A related problem, albeit in dimension $3$, is the so-called cylinder conjecture of Ball \cite{cilinder},  which is also related to the direction problem. Also, when proving stability results with the subresultant method of Sz\H{o}nyi and Weiner, it is often a crucial step to derive that a set $\cL$ of lines covers each point either less then $a$ or more than $b$ times; in other words, $\cL$ is a (dual) $[a,b]$-avoiding set \cite{SzonyiWeiner, Szonyi, Weiner}. 
	
	One may also consider the modular version of the problem. Sets of even type (that is, $0 \mod 2$ sets) have gained significant attention, which notion can be generalised to $k \mod p$ multisets \cite{Szonyi}. 
	
	To obtain sets of even type or sets of odd type of various sizes, we propose the construction below.
	
	\begin{const}
		Let $\PG(2,q)$ be of even order. Consider the set $\cH$ of hyperovals and take the symmetric difference of every element of each subset of $\cH$. This corresponds to the linear code generated by the characteristic vectors of the hyperovals. It is obvious that the obtained  point sets are sets of even type with even cardinality. 
		Take also the complement of these point sets, which will  be sets of odd type, and with odd cardinality this time.
	\end{const}
	
	We conjecture that the obtained set system shows that the modular version of Conjecture \ref{mainconj} holds in a stronger from: if $q$ is even, $k$ not too small or large, and $k\not\equiv m \pmod{2}$, then there are $k$-avoiding sets of size $m$ for every possible $m$ with all secant length different from $k\pmod{2}$. It would be interesting the describe the weight enumerator of the corresponding code and it would also be interesting to see whether there exist sets of even type which cannot be obtained in this manner \cite{Szonyi_personal}, see also \cite{Zullo}.
	
	We mention that concerning the case when $q$ is odd,  Ballister, Bollobás and  Füredi \cite{Furedi} studied the related problem of determining the
	minimum size of the symmetric differences of lines
	in projective planes, while 
	Ball and Csajbók studied sets with few odd secants \cite{Csajbok}.

	For untouchable sets or sets without a tangent (case $k=1$), computer search search suggested that on the top of the critical window, $m=2q-1$ might always be missing from $\Sp(1,q)$ if $q$ is odd. (Note that for $q=8$ and $16$, we have such an object.)
	
	The case $k=2$ is also particularly interesting. First, the possible sizes in the critical window seems hard to determine, see Table \ref{tab:smallq}. Second, if we have a $2$-avoiding set of size $k+1$ then, using their duals as seen after Proposition \ref{linesprop}, we obtain $k$-avoiding sets of size in the critical window. Here the difficulty is to find a $2$-avoiding set of size $k+1$ which has about $(k-1)q$ tangents. 
	
	Finally, we reiterate here  Problem \ref{fstaravoiding} on the existence of point sets where there is a given number $M(q)$ of integers prescribed for each line which describes the forbidden secant size of the corresponding line.

\end{document}